\documentclass[12pt, reqno]{amsart}
\usepackage{amsmath, amsthm, amscd, amsfonts, amssymb, graphicx, color}
\usepackage[bookmarksnumbered, colorlinks, plainpages]{hyperref}  %,pagebackref=true
\hypersetup{colorlinks=true,linkcolor=red, anchorcolor=green, citecolor=cyan, urlcolor=red, filecolor=magenta, pdftoolbar=true}

\newtheorem{theorem}{Theorem}[section]
\newtheorem{lemma}[theorem]{Lemma}

\newtheorem{corollary}[theorem]{Corollary}
\theoremstyle{definition}

\theoremstyle{remark}
\newtheorem{remark}[theorem]{\bf{Remark}}
\numberwithin{equation}{section}
\begin{document}
%------------------------
\address{$^{[1,3]}$ Department of Mathematics, Jadavpur University, Kolkata 700032, West Bengal, India.}
\email{\url{pintubhunia5206@gmail.com} ; \url{kalloldada@gmail.com}}

\address{$^{[2_a]}$ University of Monastir, Faculty of Economic Sciences and Management of Mahdia, Mahdia, Tunisia}
\address{$^{[2_b]}$ Laboratory Physics-Mathematics and Applications (LR/13/ES-22), Faculty of Sciences of Sfax, University of Sfax, Sfax, Tunisia}
\email{\url{kais.feki@fsegma.u-monastir.tn}\,;\,\url{kais.feki@hotmail.com}}


\date{\today}
%\date{June 24, 2019}
\author[P. Bhunia, K. Feki and K. Paul] {Pintu Bhunia$^{1}$, Kais Feki$^{2_{a,b}}$ and Kallol Paul$^{3}$ }
\title[Numerical radius inequalities]{\Small{Numerical radius inequalities for products and sums of semi-Hilbertian space operators}}

\maketitle
%------------------------
%----------\thispagestyle{empty}

\begin{abstract}
New inequalities for the $A$-numerical radius of the products and sums of operators acting on a semi-Hilbert space, i.e. a space generated
by a positive semidefinite operator $A$, are established. In particular, it is proved for operators $T$ and $S,$ having $A$-adjoint, that
$$
\omega_A(TS) \leq \frac{1}{2}\omega_A(ST)+\frac{1}{4}\Big(\|T\|_A\|S\|_A+\|TS\|_A\Big),
$$
 where $\omega_A(T)$ and $\|T\|_A$ denote the $A$-numerical radius and the $A$-operator seminorm of an operator $T$.
%Generally, $\omega_A(\cdot)$ does not satisfy the submultiplicative property.
% In particular, if we consider two operators $T, S\in\mathcal{B}_{A^{1/2}}(\mathcal{H})$ for which one is $A$-positive and other is $A$-normal
 % then $\omega_A(TS)\leq \omega_A(T)\omega_A(S).$
%Let $A$ be a positive (semi-definite) bounded linear operator on a complex Hilbert space $(\mathcal{H},\langle\cdot,\cdot\rangle)$. Let
%$\omega_A(T)$ and $\|T\|_A$ denote the $A$-numerical radius and the $A$-operator seminorm of an operator $T$ acting on the semi-Hilbert space $(\mathcal{H},\langle\cdot,\cdot\rangle_A)$ respectively, where $\langle x, y\rangle_{A} :=\langle Ax, y\rangle$ for all $x,y\in \mathcal{H}$. In this paper we establish some $A$-numerical radius inequalities for products and sums of semi-Hilbert space operators. In particular, among other inequalities, we prove that
%$$
%\omega_A(TS) \leq \frac{1}{2}\omega_A(ST)+\frac{1}{4}\Big(\|T\|_A\|S\|_A+\|TS\|_A\Big).
%$$
\end{abstract}

\section{Introduction and Preliminaries}\label{s1}
\noindent Let $\mathcal{B}(\mathcal{H})$ stand for the $C^{\ast}$-algebra of all bounded linear operators on a complex Hilbert space $\mathcal{H}$ with inner product $\langle\cdot,\cdot\rangle$ and the corresponding norm $\|\cdot\|$. For $T\in\mathcal{B}(\mathcal{H})$, we denote by $\mathcal{R}(T)$, $\mathcal{N}(T)$ and $T^*$ the range, the kernel and the adjoint of $T$, respectively. For a given linear subspace $\mathcal{M}$ of $\mathcal{H}$, its closure in the norm topology of $\mathcal{H}$ will be denoted by $\overline{\mathcal{M}}$. Further, let $P_{\mathcal{S}}$ stand for the orthogonal projection onto a closed subspace $\mathcal{S}$ of $\mathcal{H}$. An operator $T\in\mathcal{B}(\mathcal{H})$ is called positive if $\langle Tx, x\rangle\geq0$ for all $x\in{\mathcal H }$, and we then write $T\geq 0$. Furthermore, if $T\geq 0$, then the square root of $T$ is denoted by
$T^{1/2}$. For $T\in\mathcal{B}(\mathcal{H})$,  the absolute value of $T$, denoted by $|T|$,  is defined as $|T|=(T^*T)^{1/2}$. Throughout the article, $A$ denotes a non-zero positive operator on $\mathcal{H}$. The positive operator $A$ induces the following semi-inner product
$$\langle\cdot,\cdot\rangle_{A}:\mathcal{H}\times \mathcal{H}\longrightarrow\mathbb{C},\;(x,y)\longmapsto \langle x, y\rangle_{A}:=\langle Ax, y\rangle=\langle A^{1/2}x, A^{1/2}y\rangle.$$
The seminorm induced by ${\langle \cdot, \cdot\rangle}_A$ is given by ${\|x\|}_A=\|A^{1/2}x\|$ for all $x\in\mathcal{H}$. It is easy to check that ${\|\cdot\|}_A$ is a norm if and only if $A$ is injective and that the seminormed space $(\mathcal{H}, {\|\cdot\|}_A)$ is complete if and only if $\overline{\mathcal{R}(A)}=\mathcal{R}(A)$. It is well-known that the semi-inner product ${\langle \cdot, \cdot\rangle}_{A}$ induces an inner product on the quotient space $\mathcal{H}/\mathcal{N}(A)$ which is not complete unless $\mathcal{R}(A)$ is closed. However, a canonical construction due to de Branges and Rovnyak \cite{branrov} (see also \cite{fekilaa}) shows that the completion of $\mathcal{H}/\mathcal{N}(A)$ is isometrically isomorphic to the Hilbert space $\mathcal{R}(A^{1/2})$ with the inner product
\begin{align*}\label{I.1}
\langle A^{1/2}x,A^{1/2}y\rangle_{\mathbf{R}(A^{1/2})}:=\langle P_{\overline{\mathcal{R}(A)}}x, P_{\overline{\mathcal{R}(A)}}y\rangle,\quad\forall\, x,y \in \mathcal{H}.
\end{align*}
For the sequel, the Hilbert space $\big(\mathcal{R}(A^{1/2}), \langle\cdot,\cdot\rangle_{\mathbf{R}(A^{1/2})}\big)$ will be simply denoted by $\mathbf{R}(A^{1/2})$. For an account of results related to $\mathbf{R}(A^{1/2})$, we refer the readers to \cite{acg3}.

Let $T \in \mathcal{B}(\mathcal{H})$. We recall that an operator $S\in\mathcal{B}(\mathcal{H})$ is called an $A$-adjoint of $T$ if $\langle Tx, y\rangle_A=\langle x, Sy\rangle_A$ for all $x,y\in \mathcal{H}$. One can observe that the existence of an $A$-adjoint of $T$ is equivalent to the existence of a solution in $\mathcal{B}(\mathcal{H})$ of the equation $AX = T^*A$.
%If $T\in \mathcal{B}_A(\mathcal{H})$, the reduced solution of the equation
%$AX=T^*A$ is a distinguished $A$-adjoint operator of $T$, which is denoted by $T^{\sharp_A}$.
%This kind of operator equations can be studied by using Douglas theorem (see \cite{doug}) which briefly states that the equation $TX=S$ has a solution in $\mathcal{B}(\mathcal{H})$ if and only if $\mathcal{R}(S) \subseteq \mathcal{R}(T)$. This, in turn, equivalent to the existence of a constant $\lambda>0$ such that $\|S^*x\|\leq \lambda \|T^*x\|$ for all $x\in \mathcal{H}$. in addition, among its many solutions it has only one, denoted by $Q$, which satisfies $\mathcal{R}(Q) \subseteq \overline{\mathcal{R}(T^{*})}$. Such $Q$ is said the reduced solution of $TX=S$.
Clearly, the existence of an $A$-adjoint operator is not guaranteed. If the set of all operators admitting $A$-adjoints is denoted by $\mathcal{B}_{A}(\mathcal{H})$, then by Douglas theorem \cite{doug} we have,
$$\mathcal{B}_{A}(\mathcal{H})=\left\{T\in \mathcal{B}(\mathcal{H})\,;\;\mathcal{R}(T^{*}A)\subseteq \mathcal{R}(A)\right\}.$$
If $T\in \mathcal{B}_A(\mathcal{H})$, the reduced solution of the equation
$AX=T^*A$ is a distinguished $A$-adjoint operator of $T$, which is denoted by $T^{\sharp_A}$ and satisfies $\mathcal{R}(T^{\sharp_A}) \subseteq \overline{\mathcal{R}(A)}$.  Note that $T^{\sharp_A}=A^\dag T^*A$, where $A^\dag$ is the Moore-Penrose inverse of $A$ (see \cite{acg2}). If $T \in \mathcal{B}_A({\mathcal{H}})$, then $T^{\sharp_A} \in \mathcal{B}_A({\mathcal{H}})$, $(T^{\sharp_A})^{\sharp_A}=P_{\overline{\mathcal{R}(A)}}TP_{\overline{\mathcal{R}(A)}}$ and $((T^{\sharp_A})^{\sharp_A})^{\sharp_A}=T$. Moreover, if $S\in \mathcal{B}_A(\mathcal{H})$ then $TS \in\mathcal{B}_A({\mathcal{H}})$ and $(TS)^{\sharp_A}=S^{\sharp_A}T^{\sharp_A}.$
For more results concerning $T^{\sharp_A}$, we invite the readers to see \cite{acg1,acg2}.
%An operator $U\in  \mathcal{B}_A(\mathcal{H})$ is called $A$-unitary if $\|Ux\|_A=\|U^{\sharp_A}x\|_A=\|x\|_A$ for all $x\in \mathcal{H}$. It should be mention that, an operator $U\in  \mathcal{B}_A(\mathcal{H})$ is $A$-unitary if and only if $U^{\sharp_A} U=(U^{\sharp_A})^{\sharp_A} U^{\sharp_A}=P_{\overline{\mathcal{R}(A)}}$ (see \cite{acg1}).
An operator $T$ is called $A$-bounded if there exists $\lambda>0$ such that $ \|Tx\|_{A} \leq \lambda \|x\|_{A},$ for every $x\in \mathcal{H}.$  In virtue of Douglas theorem, one can see that the set  of all operators admitting $A^{1/2}$-adjoints, denoted by $\mathcal{B}_{A^{1/2}}(\mathcal{H})$, is same as the collection of all $A$-bounded operators, i.e.,
$$\mathcal{B}_{A^{1/2}}(\mathcal{H})=\left\{T \in \mathcal{B}(\mathcal{H})\,;\;\exists \,\lambda > 0\,;\;\|Tx\|_{A} \leq \lambda \|x\|_{A},\;\forall\,x\in \mathcal{H}\right\}.$$
It is well-known that $\mathcal{B}_{A}(\mathcal{H})$ and $\mathcal{B}_{A^{1/2}}(\mathcal{H})$ are two subalgebras of $\mathcal{B}(\mathcal{H})$ which are, in general, neither closed nor dense in $\mathcal{B}(\mathcal{H})$. Further, we have $\mathcal{B}_{A}(\mathcal{H})\subseteq \mathcal{B}_{A^{1/2}}(\mathcal{H})$ (see \cite{acg1,feki01}). If $T\in\mathcal{B}_{A^{1/2}}(\mathcal{H})$, then the seminorm of $T$ induces by $\langle\cdot,\cdot\rangle_{A}$ is given by
\begin{equation}\label{semii}
\|T\|_A:=\sup_{\substack{x\in \overline{\mathcal{R}(A)},\\ x\not=0}}\frac{\|Tx\|_A}{\|x\|_A}=\sup\big\{{\|Tx\|}_A\,; \,\,x\in \mathcal{H},\, {\|x\|}_A =1\big\}<\infty.
\end{equation}
It was shown in \cite[Proposition 2.3]{acg2} that, for every $T \in \mathcal{B}_A({\mathcal{H}})$, we have
\begin{equation}\label{diez}
\|T^{\sharp_A}T\|_A = \| TT^{\sharp_A}\|_A=\|T\|_A^2 =\|T^{\sharp_A}\|_A^2.
\end{equation}
%Notice that it was proved in \cite{fg} that for $T\in\mathcal{B}_{A^{1/2}}(\mathcal{H})$ we have
%\begin{equation}\label{newsemi}
%\|T\|_A=\sup\left\{|\langle Tx, y\rangle_A|\,;\;x,y\in \mathcal{H},\,\|x\|_{A}=\|y\|_{A}= 1\right\}.
%\end{equation}
Furthermore, Saddi \cite{saddi} introduced   the concept of the $A$-numerical radius of an operator $T\in\mathcal{B}(\mathcal{H})$ as follows
\begin{align*}
\omega_A(T)
&:= \sup \left\{|\langle Tx, x\rangle_A|\,;\;x\in\mathcal{H},\|x\|_A = 1\right\}.
\end{align*}
We mention here that it may happen that ${\|T\|}_A$ and $\omega_A(T)$ are equal to $+ \infty$ for some $T\in\mathcal{B}(\mathcal{H})\setminus \mathcal{B}_{A^{1/2}}(\mathcal{H})$ (see \cite{feki01}). However, it was shown in \cite{bakfeki01} that ${\|\cdot\|}_A$ and $\omega_A(\cdot)$ are equivalent seminorms on $\mathcal{B}_{A^{1/2}}(\mathcal{H})$. More precisely, for every $T\in \mathcal{B}_{A^{1/2}}(\mathcal{H})$, the following inequalities hold
\begin{equation}\label{refine1}
\tfrac{1}{2} \|T\|_A\leq\omega_A(T) \leq \|T\|_A.
\end{equation}
%It should be emphasized that it may happen that ${\|T\|}_A = + \infty$ for some $T\in\mathcal{B}(\mathcal{H})$ (see \cite{feki01}).
Let $T\in\mathcal{B}_{A^{1/2}}(\mathcal{H})$. Then $\|T\|_A=0$ if and only if $AT=0$. Furthermore, $\|Tx\|_A\leq \|T\|_A\|x\|_A,$ for every $x\in \mathcal{H}.$ % it holds that
%\begin{equation}\label{semiiineq}
%\|Tx\|_A\leq \|T\|_A\|x\|_A,\;\forall\,x\in \mathcal{H}.
%\end{equation}
 This implies that for all $T,S\in \mathcal{B}_{A^{1/2}}(\mathcal{H})$, $\|TS\|_A\leq \|T\|_A\|S\|_A.$ An operator $T\in\mathcal{B}(\mathcal{H})$ is called $A$-selfadjoint if $AT$ is selfadjoint and it is called $A$-positive if $AT$ is a positive operator. An operator $T\in \mathcal{B}_A(\mathcal{H})$ is said to be $A$-normal if $T^{\sharp_A}T=TT^{\sharp_A}$ (see \cite{acg3}). For the sequel, if $A=I$ then $\|T\|$, $r(T)$ and $\omega(T)$ denote respectively the classical operator norm, the spectral radius and the numerical radius of an operator $T$.

%Notice that if $T\in\mathcal{B}_{A^{1/2}}(\mathcal{H})$ and satisfies $AT^2=0$, then by \cite[Corollary 2]{feki01} we have
%\begin{equation}\label{at2}
%\omega_A(T)= \frac{1}{2}\|T\|_A.
%\end{equation}
%Notice that the $A$-numerical radius of semi-Hilbertian space operators satisfies the weak $A$-unitary invariance property which asserts that
%\begin{equation}\label{weak}
%\omega_{A}(U^{\sharp}TU)=\omega_{A}(T),
%\end{equation}
%for every $T\in \mathcal{B}_{A^{1/2}}(\mathcal{H})$ and every $A$-unitary operator $U\in \mathcal{B}_{A}(\mathcal{H})$ (see \cite[Lemma 3.8]{bfeki}).
For any operator $T\in {\mathcal B}_A({\mathcal H})$, we write $\Re_A(T):=\frac{T+T^{\sharp_A}}{2}$. Following \cite[Theorem 2.5]{zamani1}, we have that if $T\in\mathcal{B}_{A}(\mathcal{H})$ then
\begin{align}\label{zm}
\omega_A(T) = \displaystyle{\sup_{\theta \in \mathbb{R}}}{\left\|\Re_A(e^{i\theta}T)\right\|}_A.
\end{align}
The $A$-spectral radius of an operator $T\in \mathcal{B}_{A^{1/2}}(\mathcal{H})$ was defined by the second author in \cite{feki01} as
\begin{equation}\label{newrad}
r_A(T):=\displaystyle\inf_{n\geq 1}\|T^n\|_A^{\frac{1}{n}}=\displaystyle\lim_{n\to\infty}\|T^n\|_A^{\frac{1}{n}}.
\end{equation}
The second equality in \eqref{newrad} is also proved in \cite[Theorem 1]{feki01}. In addition, it was shown in \cite{feki01} that $r_A(\cdot)$ satisfies the commutativity property, i.e.,
\begin{equation}\label{commut}
r_A(TS)=r_A(ST), \quad\forall\;T,S\in \mathcal{B}_{A^{1/2}}(\mathcal{H}).
\end{equation}
Also, the following relation between the $A$-spectral radius and the $A$-numerical radius of $A$-bounded operators is also proved in \cite{feki01}.
\begin{equation}\label{rawa}
r_A(T)\leq \omega_A(T), \quad\forall\;T\in \mathcal{B}_{A^{1/2}}(\mathcal{H}).
\end{equation}
%for all $T,S\in \mathcal{B}_{A^{1/2}}(\mathcal{H})$.
%The reader is invited to see \cite{bakfeki01,bakfeki04,bfeki,feki03,fekilaa,faiot,fekisidha2019,zamani2,tamzhang,zamani1} and the references therein.
%Very recently, several inequalities for the $\mathbb{A}$-numerical radius of $2 \times 2$ operator matrices have been established by P. Bhunia et al. (see \cite{BPN}). This paper is devoted also to prove several new $\mathbb{A}$-numerical radius inequalities of certain $2 \times 2$ operator matrices. Some of the obtained results cover and extend the following works \cite{pinar,HirKit,S}.

Recently, many mathematicians have obtained  different $A$-numerical radius inequalities of semi-Hilbertian space operators, the interested readers are invited to see \cite{bfeki, BPN, RajAOT, feki03, fekip1,rout1,rout2} and references therein. Here, we obtain several new inequalities for the $A$-numerical radius of the products and the sums of semi-Hilbertian space operators. The bounds obtained here improve on the existing bounds.

\section{\textbf{On inequalities for product of operators}}\label{s2}

We begin this section with the following known lemma which can be found in \cite{feki01}.

\begin{lemma} \label{lem03}
Let $T\in\mathcal{B}(\mathcal{H})$ be an $A$-selfadjoint operator. Then,
\begin{equation*}
\|T\|_{A}=\omega_A(T)=r_A(T).
\end{equation*}
\end{lemma}

Our first result reads as:

\begin{theorem}\label{theo2prod}
Let $T, S\in\mathcal{B}_{A}(\mathcal{H})$. Then,
\begin{align*}
\omega_A(TS) \leq \|T\|_A\,\omega_A(S)+\frac{1}{2}\omega_A\left( TS \pm ST^{\sharp_A}\right).
\end{align*}
\end{theorem}

\begin{proof}
Let $\theta \in \mathbb{R}$. Clearly, $\Re_A(e^{i \theta}TS)$ is an $A$-selfadjoint operator. Therefore, from Lemma \ref{lem03} we get,
\begin{eqnarray*}
\left \|\Re_A(e^{i \theta}TS)\right\|_A &=& \omega_A\left(\Re_A(e^{i \theta}TS)\right)\\
&=& \omega_A\left (\frac{1}{2}(e^{i \theta}TS+e^{-i \theta}S^{\sharp_A}T^{\sharp_A} ) \right )\\
&=& \omega_A\left (\frac{1}{2}(e^{i \theta}TS+ e^{-i \theta}TS^{\sharp_A}+e^{-i \theta}S^{\sharp_A}T^{\sharp_A}-e^{-i \theta}TS^{\sharp_A} ) \right )\\
&=& \omega_A\left ( T\Re_A (e^{i \theta}S) +  \frac{1}{2}e^{-i \theta}( S^{\sharp_A}T^{\sharp_A}- TS^{\sharp_A} ) \right )\\
&\leq & \omega_A \left ( T\Re_A (e^{i \theta}S) \right )+ \omega_A\left ( \frac{1}{2}e^{-i \theta}( S^{\sharp_A}T^{\sharp_A}- TS^{\sharp_A} ) \right )\\
&\leq& \left \|T\Re_A (e^{i \theta}S) \right \|_A+ \frac{1}{2}\omega_A\left ( S^{\sharp_A}T^{\sharp_A}- TS^{\sharp_A}  \right )\\
&\leq&  \|T\|_A \left \|\Re_A (e^{i \theta}S) \right \|_A+ \frac{1}{2}\omega_A\left ( S^{\sharp_A}T^{\sharp_A}- TS^{\sharp_A}  \right )\\
&\leq& \|T\|_A\omega_A(S)+\frac{1}{2}\omega_A\left( S^{\sharp_A}T^{\sharp_A}- TS^{\sharp_A}\right).
\end{eqnarray*}
Taking supremum over all $\theta \in \mathbb{R}$ we get,
\begin{align}\label{prod}
\omega_A(TS) \leq \|T\|_A\omega_A(S)+\frac{1}{2}\omega_A\left( S^{\sharp_A}T^{\sharp_A}- TS^{\sharp_A}\right).
\end{align}
Next, for   $x\in \mathcal{H}$ we have,
\[ |\langle(S^{\sharp_A}T^{\sharp_A}- TS^{\sharp_A})x,x \rangle_A |=|\langle (TS-ST^{\sharp_A})x,x \rangle_A|.\]
This implies that $\omega_A\left ( S^{\sharp_A}T^{\sharp_A}- TS^{\sharp_A}  \right )=\omega_A\left( TS-ST^{\sharp_A}\right).$ Thus, it follows from (\ref{prod}) that
\begin{align}\label{kjjj}
\omega_A(TS) \leq \|T\|_A\omega_A(S)+\frac{1}{2}\omega_A\left( TS-ST^{\sharp_A}\right).
\end{align}
Also, replacing $T$ by $iT$ in \eqref{kjjj} we get,
\begin{align}\label{kjjjj}
\omega_A(TS) \leq \|T\|_A\omega_A(S)+\frac{1}{2}\omega_A\left( TS+ST^{\sharp_A}\right).
\end{align}
The proof now follows from \eqref{kjjj} and  \eqref{kjjjj}.
\end{proof}

\begin{remark}
It is easy to verify that (also see in \cite[Cor. 3.3]{BPN}) $$\omega_A\left( TS \pm ST^{\sharp_A}\right)\leq 2\|T\|_A\omega_A(S).$$ Therefore,
$\|T\|_A\omega_A(S)+\frac{1}{2}\omega_A\left( TS \pm ST^{\sharp_A}\right)\leq 2\|T\|_A\omega_A(S).$
Thus, the inequality obtained in Theorem \ref{theo2prod} is stronger than the well-know inequality $$\omega_A(TS) \leq 2\|T\|_A\omega_A(S).$$

\end{remark}

In order to obtain our next inequality that gives an upper bound for the $A$-numerical radius of product of two operators, we need the following lemmas. First we consider the ${2\times 2}$ operator diagonal matrix $\mathbb{A}=\begin{pmatrix}
A &0\\
0 &A
\end{pmatrix}$. Clearly, $\mathbb{A}$ is a positive operator on $\mathcal{H}\oplus \mathcal{H}$. So, $\mathbb{A}$ induces the following semi-inner product on  $ \mathcal{H}\oplus \mathcal{H} $  defined as
$$\langle x, y\rangle_{\mathbb{A}}= \langle \mathbb{A}x, y\rangle=\langle x_1, y_1\rangle_A+\langle x_2, y_2\rangle_A,$$
 for all $x=(x_1,x_2), y=(y_1,y_2)\in \mathcal{H}\oplus \mathcal{H}$. Note that if $T,S,X,Y \in \mathcal{B}_{A}(\mathcal{H})$ then, it was shown in \cite[Lemma 3.1]{bfeki} that $\begin{pmatrix}
T & S \\
X&Y
\end{pmatrix}\in \mathcal{B}_{\mathbb{A}}(\mathcal{H}\oplus \mathcal{H})$ and
\begin{equation}\label{diez2}
\begin{pmatrix}
T&S \\
X&Y
\end{pmatrix}^{\sharp_\mathbb{A}}=\begin{pmatrix}
T^{\sharp_A} & X^{\sharp_A}\\
S^{\sharp_A} & Y^{\sharp_A}
\end{pmatrix}.
\end{equation}

\begin{lemma}$($\cite{fekip1}$)$\label{lem01}
Let $T,S\in \mathcal{B}(\mathcal{H})$ be $A$-positive operators. % and let $\mathbb{A}=\begin{pmatrix}A&0\\0&A\end{pmatrix}$.
 Then,
\begin{equation}
\omega_{\mathbb{A}}\left[\begin{pmatrix}0&T\\S&0\end{pmatrix}\right]= \frac{1}{2}\left\|T+S\right\|_A.
\end{equation}
\end{lemma}

\begin{lemma}$($\cite{faiot}$)$\label{lem02}
Let $T,S\in \mathcal{B}_{A^{1/2}}(\mathcal{H})$. % and let $\mathbb{A}=\begin{pmatrix}
%A &0\\
%0 &A
%\end{pmatrix}$.
 Then,
\begin{itemize}
  \item [(a)] $\omega_{\mathbb{A}}\left[\begin{pmatrix}
T&0\\
0 &S
\end{pmatrix}\right]=\max\{\omega_A(T),\omega_A(S)\}.$
In particular,
\begin{equation}\label{zm1}
\omega_{\mathbb{A}}\left[\begin{pmatrix}
T&0\\
0 &T
\end{pmatrix}\right]=\omega_{\mathbb{A}}\left[\begin{pmatrix}
T&0\\
0 &T^{\sharp_A}
\end{pmatrix}\right]=\omega_A(T).
\end{equation}
  \item [(b)] $\left\|\begin{pmatrix}
0&T\\
S &0
\end{pmatrix}\right\|_{\mathbb{A}}=\left\|\begin{pmatrix}
T&0\\
0 &S
\end{pmatrix}\right\|_{\mathbb{A}}=\max\left\{\|T\|_A,\|S\|_A\right\}.$
\end{itemize}
\end{lemma}

Now we are in a position to obtain the following inequality.
\begin{theorem}\label{thk}
Let $T, S\in\mathcal{B}_{A}(\mathcal{H})$. Then,
\begin{align}\label{r1}
\omega_A(TS) \leq \frac{1}{2}\omega_A(ST)+\frac{1}{4}\Big(\|T\|_A\|S\|_A+\|TS\|_A\Big).
\end{align}
\end{theorem}

\begin{proof}
Let $\theta\in \mathbb{R}$. Since $\Re_A(e^{i \theta}TS)$ is an $A$-selfadjoint operator, so by Lemma \ref{lem03} we have,
\begin{align}\label{18mar}
\|\Re_A(e^{i \theta}TS)\|_A
& =r_A[\Re_A(e^{i \theta}TS)] \nonumber\\
 &=\tfrac{1}{2}r_A(e^{i \theta}TS+e^{-i \theta}S^{\sharp_A}T^{\sharp_A}).
\end{align}
On the other hand, we have
\begin{align*}
r_A(e^{i \theta}TS+e^{-i \theta}S^{\sharp_A}T^{\sharp_A})
& =r_{\mathbb{A}}\left[ \begin{pmatrix}
e^{i \theta}TS+e^{-i \theta}S^{\sharp_A}T^{\sharp_A} &0\\
0 &0
\end{pmatrix}\right] \\
 &=r_{\mathbb{A}}\left[ \begin{pmatrix}
e^{i \theta}T &S^{\sharp_A}\\
0 &0
\end{pmatrix}
\begin{pmatrix}
S&0\\
e^{-i \theta}T^{\sharp_A}  &0
\end{pmatrix}
\right] \\
 &=r_{\mathbb{A}}\left[
\begin{pmatrix}
S&0\\
e^{-i \theta}T^{\sharp_A}  &0
\end{pmatrix}
\begin{pmatrix}
e^{i \theta}T &S^{\sharp_A}\\
0 &0
\end{pmatrix}
\right] \quad (\text{ by }\; \eqref{commut}) \\
&=r_{\mathbb{A}}\left[
\begin{pmatrix}
e^{i \theta}ST&SS^{\sharp_A}\\
T^{\sharp_A}T  &e^{-i \theta}T^{\sharp_A}S^{\sharp_A}
\end{pmatrix}\right].
\end{align*}
By applying \eqref{rawa} we get
\begin{align*}
r_A(e^{i \theta}TS+e^{-i \theta}S^{\sharp_A}T^{\sharp_A})
&\leq\omega_{\mathbb{A}}\left[
\begin{pmatrix}
e^{i \theta}ST&SS^{\sharp_A}\\
T^{\sharp_A}T  &e^{-i \theta}T^{\sharp_A}S^{\sharp_A}
\end{pmatrix}\right]\\
&\leq\omega_{\mathbb{A}}\left[
\begin{pmatrix}
e^{i \theta}ST&0\\
0  &e^{-i \theta}T^{\sharp_A}S^{\sharp_A}
\end{pmatrix}\right]
+\omega_{\mathbb{A}}\left[
\begin{pmatrix}
0&SS^{\sharp_A}\\
T^{\sharp_A}T  &0
\end{pmatrix}\right]\\
&=\omega_A(ST)+\frac{1}{2}\|SS^{\sharp_A}+T^{\sharp_A}T\|_A,
\end{align*}
where the last equality follows from Lemma \ref{lem01} together with \eqref{zm1}. Therefore, from \eqref{18mar} we get,
\begin{align*}
\|\Re_A(e^{i \theta}TS)\|_A\leq \tfrac{1}{2}\omega_A(ST)+\tfrac{1}{4}\|SS^{\sharp_A}+T^{\sharp_A}T\|_A.
\end{align*}
 Hence, by taking supremum over all $\theta\in \mathbb{R}$ and then using \eqref{zm} we get,
\begin{align}\label{subst2}
\omega_A(TS)\leq \tfrac{1}{2}\omega_A(ST)+\tfrac{1}{4}\|SS^{\sharp_A}+T^{\sharp_A}T\|_A,
\end{align}
If $AT=0$ or $AS=0$, then the inequality \eqref{r1} holds trivially. Assume that $AT\neq0$ and $AS\neq0$. By Replacing $T$ and $S$ by $\sqrt{\frac{\|S\|_A}{\|T\|_A}}T$ and $\sqrt{\frac{\|T\|_A}{\|S\|_A}}S$, respectively, in \eqref{subst2} we obtain,
\begin{align}\label{jjjj}
\omega_A(TS)\leq \frac{1}{2}\omega_A(ST)+\frac{1}{4}\left\|\tfrac{\|S\|_A}{\|T\|_A}T^{\sharp_A}T+\tfrac{\|T\|_A}{\|S\|_A}SS^{\sharp_A}\right\|_A.
\end{align}
It is easy to see  that the operator $\tfrac{\|S\|_A}{\|T\|_A}T^{\sharp_A}T+\tfrac{\|T\|_A}{\|S\|_A}SS^{\sharp_A}$ is  $A$-positive. So, an application of Lemma \ref{lem03} gives
\begin{equation}\label{bib1}
\left\|\tfrac{\|S\|_A}{\|T\|_A}T^{\sharp_A}T+\tfrac{\|T\|_A}{\|S\|_A}SS^{\sharp_A}\right\|_A=
r_A\left(\tfrac{\|S\|_A}{\|T\|_A}T^{\sharp_A}T+\tfrac{\|T\|_A}{\|S\|_A}SS^{\sharp_A}\right).
\end{equation}
Next,
\begin{align*}
r_A\left(\tfrac{\|S\|_A}{\|T\|_A}T^{\sharp_A}T+\tfrac{\|T\|_A}{\|S\|_A}SS^{\sharp_A}\right)
& =r_{\mathbb{A}}\left[ \begin{pmatrix}
\tfrac{\|S\|_A}{\|T\|_A}T^{\sharp_A}T+\tfrac{\|T\|_A}{\|S\|_A}SS^{\sharp_A} &0\\
0 &0
\end{pmatrix}\right] \nonumber\\
 &=r_{\mathbb{A}}\left[ \begin{pmatrix}
\sqrt{\tfrac{\|S\|_A}{\|T\|_A}}T^{\sharp_A} &\sqrt{\tfrac{\|T\|_A}{\|S\|_A}}S\\
0 &0
\end{pmatrix} \begin{pmatrix}
\sqrt{\tfrac{\|S\|_A}{\|T\|_A}}T&0\\
\sqrt{\tfrac{\|T\|_A}{\|S\|_A}}S^{\sharp_A} &0
\end{pmatrix}\right]\nonumber\\
\end{align*}
Further, by applying \eqref{commut}, we get
\begin{align}\label{bib2}
r_A\left(\tfrac{\|S\|_A}{\|T\|_A}T^{\sharp_A}T+\tfrac{\|T\|_A}{\|S\|_A}SS^{\sharp_A}\right)
 &=r_{\mathbb{A}}\left[ \begin{pmatrix}
\sqrt{\tfrac{\|S\|_A}{\|T\|_A}}T&0\\
\sqrt{\tfrac{\|T\|_A}{\|S\|_A}}S^{\sharp_A} &0
\end{pmatrix}\begin{pmatrix}
\sqrt{\tfrac{\|S\|_A}{\|T\|_A}}T^{\sharp_A} &\sqrt{\tfrac{\|T\|_A}{\|S\|_A}}S\\
0 &0
\end{pmatrix} \right] \nonumber\\
& =r_{\mathbb{A}}\left[\begin{pmatrix}
\tfrac{\|S\|_A}{\|T\|_A}TT^{\sharp_A} &  TS\\
S^{\sharp_A}T^{\sharp_A}    &  \tfrac{\|T\|_A}{\|S\|_A}S^{\sharp_A}S
\end{pmatrix}\right].
\end{align}
In addition, we see that
\begin{align*}
\mathbb{A}\begin{pmatrix}
\tfrac{\|S\|_A}{\|T\|_A}TT^{\sharp_A} &  TS\\
S^{\sharp_A}T^{\sharp_A}    &  \tfrac{\|T\|_A}{\|S\|_A}S^{\sharp_A}S
\end{pmatrix}
& = \begin{pmatrix}
\tfrac{\|S\|_A}{\|T\|_A}ATT^{\sharp_A} &  ATS\\
AS^{\sharp_A}T^{\sharp_A}    &  \tfrac{\|T\|_A}{\|S\|_A}AS^{\sharp_A}S
\end{pmatrix}\\
 &= \begin{pmatrix}
\tfrac{\|S\|_A}{\|T\|_A}(TT^{\sharp_A})^*A & (S^{\sharp_A}T^{\sharp_A})^*A \\
(TS)^*A    &  \tfrac{\|T\|_A}{\|S\|_A}(S^{\sharp_A}S)^*A
\end{pmatrix}\\
 &=\begin{pmatrix}
\tfrac{\|S\|_A}{\|T\|_A}TT^{\sharp_A} &  TS\\
S^{\sharp_A}T^{\sharp_A}    &  \tfrac{\|T\|_A}{\|S\|_A}S^{\sharp_A}S
\end{pmatrix}^*\mathbb{A}.
\end{align*}
This implies that, $\begin{pmatrix}
\tfrac{\|S\|_A}{\|T\|_A}TT^{\sharp_A} &  TS\\
S^{\sharp_A}T^{\sharp_A}    &  \tfrac{\|T\|_A}{\|S\|_A}S^{\sharp_A}S
\end{pmatrix}$ is an $\mathbb{A}$-selfadjoint operator. Hence, in view of Lemma \ref{lem03} we have,
\begin{equation}\label{bib3}
\left\|\begin{pmatrix}
\tfrac{\|S\|_A}{\|T\|_A}TT^{\sharp_A} &  TS\\
S^{\sharp_A}T^{\sharp_A}    &  \tfrac{\|T\|_A}{\|S\|_A}S^{\sharp_A}S
\end{pmatrix}\right\|_{\mathbb{A}}=r_{\mathbb{A}}\left[\begin{pmatrix}
\tfrac{\|S\|_A}{\|T\|_A}TT^{\sharp_A} &  TS\\
S^{\sharp_A}T^{\sharp_A}    &  \tfrac{\|T\|_A}{\|S\|_A}S^{\sharp_A}S
\end{pmatrix} \right].
\end{equation}
So, it follows from \eqref{bib1}, \eqref{bib2} and \eqref{bib3} that
$$
\left\|\tfrac{\|S\|_A}{\|T\|_A}T^{\sharp_A}T+\tfrac{\|T\|_A}{\|S\|_A}SS^{\sharp_A}\right\|_A
= \left\|\begin{pmatrix}
\tfrac{\|S\|_A}{\|T\|_A}TT^{\sharp_A} &  TS\\
S^{\sharp_A}T^{\sharp_A}    &  \tfrac{\|T\|_A}{\|S\|_A}S^{\sharp_A}S
\end{pmatrix}\right\|_{\mathbb{A}}.
$$
Finally, by applying the triangle inequality and then using Lemma \ref{lem02}, we get
\begin{align*}
\left\|\tfrac{\|S\|_A}{\|T\|_A}T^{\sharp_A}T+\tfrac{\|T\|_A}{\|S\|_A}SS^{\sharp_A}\right\|_A
 &\leq\left\|\begin{pmatrix}
\tfrac{\|S\|_A}{\|T\|_A}TT^{\sharp_A} & 0\\
0    &  \tfrac{\|T\|_A}{\|S\|_A}S^{\sharp_A}S
\end{pmatrix}\right\|_{\mathbb{A}}
+\left\|\begin{pmatrix}
0 &  TS\\
S^{\sharp_A}T^{\sharp_A}    &  0
\end{pmatrix}\right\|_{\mathbb{A}}\\
 &=\max\left\{\tfrac{\|S\|_A}{\|T\|_A}\|TT^{\sharp_A}\|_A,\tfrac{\|T\|_A}{\|S\|_A}\|S^{\sharp_A}S\|_A \right\}+\|TS\|_A\\
  &=\|S\|_A\| T\|_A+\|TS\|_A.
\end{align*}
Therefore, we get \eqref{r1} as desired by taking \eqref{jjjj} into account.
\end{proof}
\begin{remark}
By taking $A=I$ in Theorem \ref{thk} we get a recent result proved by Kittaneh et al. in \cite{A.K.1}.
\end{remark}

The following corollary is an immediate consequence of Theorem \ref{thk}.
\begin{corollary}
Let $T, S\in\mathcal{B}_{A}(\mathcal{H})$. Then
\begin{align}\label{rcc1}
\omega_A(TS) \leq \frac{1}{2}\left(\omega_A(ST)+\|T\|_A\|S\|_A\right).
\end{align}
\end{corollary}

Next, we obtain the following inequalities assuming $T$ to be $A$-positive.

\begin{theorem}\label{theo3prod}
Let $T, S\in\mathcal{B}_{A^{1/2}}(\mathcal{H})$. If $T$ is $A$-positive, then
\begin{align*}
\omega_A(TS) \leq \|T\|_A  \omega_A(S)~~\,\,\,\mbox{and}~~\,\,\,\omega_A(ST) \leq \|T\|_A  \omega_A(S).
\end{align*}
%and
%\begin{align*}
%\omega_A(ST) \leq \|T\|_A  \omega_A(S).
%\end{align*}

%\omega_A(TS) \leq \|T\|_A \bigg( (1-\alpha)\|S\|_A+\alpha \omega_A(S) \bigg), for all $\alpha\in [0,1]$. In particular,
\end{theorem}

\begin{proof}
For all $\alpha\in [0,1]$ we have,
\begin{eqnarray*}
\omega_A(TS)&=&\omega_A\left((T-\alpha \|T\|_A I)S+ \alpha \|T\|_A S\right)\\
&\leq&\omega_A\left((T-\alpha \|T\|_A I)S\right)+ \alpha \|T\|_A \omega_A(S)\\
&\leq&\left \| (T-\alpha \|T\|_A I)S\right\|_A+ \alpha \|T\|_A \omega_A(S)\\
&\leq&\left \| T-\alpha \|T\|_A I \right \|_A \| S \|_A+ \alpha \|T\|_A \omega_A(S).
\end{eqnarray*}
Since $T$ is $A$-positive, so we observe that $\left \| T-\alpha \|T\|_A I \right \|_A=(1-\alpha)\|T\|_A$ for all $\alpha\in [0,1]$. Therefore,
\begin{align}\label{pjjj}
\omega_A(TS) \leq \|T\|_A \bigg( (1-\alpha)\|S\|_A+\alpha \omega_A(S) \bigg).
\end{align}
This holds for all $\alpha\in [0,1]$, so
considering $\alpha=1$ in (\ref{pjjj}) we get, $$\omega_A(TS) \leq \|T\|_A  \omega_A(S).$$ Similarly, we can prove that
$$\omega_A(ST) \leq \|T\|_A  \omega_A(S).$$
Thus, we complete the proof.
\end{proof}

Considering $A=I$ in Theorem \ref{theo3prod} we get the following numerical radius inequalities for the product of Hilbert space operators.

\begin{corollary}\label{c}
Let $T,S \in \mathcal{B}(\mathcal{H})$ with $T$ positive. Then,
$$\omega(TS) \leq \|T\|  \omega(S) \,\,\,~~\mbox{and} ~~\,\,\, \omega(ST) \leq \|T\|  \omega(S).$$
\end{corollary}

\begin{remark}
1. We would like to note that the numerical radius $\omega(.)$ satisfies  $\omega(TS)\leq \omega(T)\omega(S)$ if either $T$ or $S$ is positive. \\
2. Abu-Omar and Kittaneh in \cite[Cor. 2.6]{A-OK} obtained that if $T,S \in \mathcal{B}(\mathcal{H})$ with $T$ positive, then $\omega(TS) \leq \frac{3}{2}\|T\|  \omega(S)$. Thus, Corollary \ref{c} is stronger than \cite[Cor. 2.6]{A-OK}.

\end{remark}

\section{\textbf{On inequalities for sum of operators}}\label{s2}

We begin this section with the following lemma.

\begin{lemma}\label{leme1}
For any $x,y,z\in \mathcal{H}$, we have
\begin{align}\label{n03}
|\langle x, y\rangle_A|^2 + |\langle x, z\rangle_A|^2
\leq \|x\|_A^2\Big(\max\{\|y\|_A^2, \|z\|_A^2\} + |\langle y, z\rangle_A|\Big).
\end{align}
\end{lemma}

\begin{proof}
First note that, by the proof of \cite[Th. 3]{dra1} we have,
\begin{align}\label{zzkk22}
|\langle x, y\rangle|^2 + |\langle x, z\rangle|^2
\leq \|x\|^2\Big(\max\{\|y\|^2, \|z\|^2\} + |\langle y, z\rangle|\Big),
\end{align}
for every $x, y, z\in \mathcal{H}$. Now,
\begin{align*}
|\langle x, y\rangle_A|^2 + |\langle x, z\rangle_A|^2 & = |\langle A^{1/2}x, A^{1/2}y\rangle|^2 + |\langle A^{1/2}x, A^{1/2}z\rangle|^2.
\end{align*}
So, by applying \eqref{zzkk22}, we obtain
\begin{align*}
 |\langle x, y\rangle_A|^2 + |\langle x, z\rangle_A|^2 & \leq \|A^{1/2}x\|^2\Big(\max\{\|A^{1/2}y\|^2, \|A^{1/2}z\|^2\} + |\langle A^{1/2}y, A^{1/2}z\rangle|\Big).
\end{align*}
Hence, we get \eqref{n03} as required.
\end{proof}

Now, we are in a position to prove the following theorem.
\begin{theorem}\label{mai9}
Let $T,S \in \mathcal{B}_A(\mathcal{H})$. Then
\begin{align*}
&\omega_A(T+S)\\
& \le \sqrt{ \frac{1}{2}\Big( \left\| TT^{\sharp_A}+SS^{\sharp_A} \right\|_A+\left\| TT^{\sharp_A}-SS^{\sharp_A} \right\|_A \Big)+\omega_A\left( ST^{\sharp_A} \right)+2\omega_A\left(T\right)\omega_A\left(S\right)}.
\end{align*}
\end{theorem}
\begin{proof}
Recall first that for every $t,s\in \mathbb{R}$ it holds
\begin{equation}\label{r}
\max\{t,s\}=\frac{1}{2}\Big(t+s+|t-s|\Big).
\end{equation}
Now, let $x\in \mathcal{H}$ with $\|x\|_A=1$. Using Lemma \ref{leme1} we get,
\begin{align*}
|\langle (T&+S)x, x\rangle_A|^2\\
&\leq |\langle x, T^{\sharp_A} x\rangle_A|^2+|\langle x, S^{\sharp_A} x\rangle_A|^2+2|\langle Tx, x\rangle_A|\,|\langle Sx, x\rangle_A|\\
&\leq \max\left\{\|T^{\sharp_A} x\|_A^2, \|S^{\sharp_A} x\|_A^2\right\}+|\langle ST^{\sharp_A} x, x\rangle_A|+2|\langle Tx, x\rangle_A|\,|\langle Sx, x\rangle_A|\\
&=\frac{1}{2}\Big(\|T^{\sharp_A} x\|_A^2+ \|S^{\sharp_A} x\|_A^2+\left|\|T^{\sharp_A} x\|_A^2-\|S^{\sharp_A} x\|_A^2 \right|\Big)+|\langle ST^{\sharp_A} x, x\rangle_A|\\
&\quad\quad\quad\quad\quad+2|\langle Tx, x\rangle_A|\,|\langle Sx, x\rangle_A|\quad (\text{by }\,\eqref{r})\\
 &=\frac{1}{2}\Big(\langle (T T^{\sharp_A}+S S^{\sharp_A})x, x\rangle_A+\left|\langle (T T^{\sharp_A}-S S^{\sharp_A})x, x\rangle_A \right|\Big)+|\langle ST^{\sharp_A} x, x\rangle_A|\\
&\quad\quad\quad\quad\quad+2|\langle Tx, x\rangle_A|\,|\langle Sx, x\rangle_A|\\
  &\leq  \frac{1}{2}\Big(\omega_A(T T^{\sharp_A}+S S^{\sharp_A})+\omega_A(T T^{\sharp_A}-S S^{\sharp_A})\Big)+\omega_A(ST^{\sharp_A})+2\omega_A\left(T\right)\omega_A\left(S\right)\\
 &=\frac{1}{2}\Big(\|T T^{\sharp_A}+S S^{\sharp_A}\|_A+\|T T^{\sharp_A}-S S^{\sharp_A}\|_A\Big)+\omega_A(ST^{\sharp_A})+2\omega_A\left(T\right)\omega_A\left(S\right),
\end{align*}
where the last equality follows from Lemma \ref{lem03}, since the operators $TT^{\sharp_A} \pm SS^{\sharp_A}$ are $A$-selfadjoint. So, we infer that
\begin{align*}
|\langle (T&+S)x, x\rangle_A|^2\\
&\leq  \frac{1}{2}\Big(\|T T^{\sharp_A}+S S^{\sharp_A}\|_A+\|T T^{\sharp_A}-S S^{\sharp_A}\|_A\Big)+\omega_A(ST^{\sharp_A})+2\omega_A\left(T\right)\omega_A\left(S\right).
\end{align*}
 Therefore, the desired result follows by taking supremum over all $x\in \mathcal{H}$ with $\|x\|_A=1$ in the last inequality.
\end{proof}

Our next objective is to refine the triangle inequality related to $\omega_A(\cdot)$. To do this, we need to recall  the following lemma from \cite{fekimjom}.

\begin{lemma}\label{main2jd}
Let $T_1,T_2,S_1,S_2\in \mathcal{B}_{A^{1/2}}(\mathcal{H})$. Then,
\begin{align}\label{eq1}
r_A\left(T_1S_1+T_2S_2\right)
& \leq \left\| \begin{pmatrix}
\|S_1T_1\|_A &\sqrt{\left\Vert S_1T_2\right\Vert_A \left\Vert S_2T_1\right\Vert_A}\\
\sqrt{\left\Vert S_1T_2\right\Vert_A \left\Vert S_2T_1\right\Vert_A} &\|S_2T_2\|_A
\end{pmatrix}\right\|.
\end{align}%
\end{lemma}

Now, we are in a position to prove the following theorem which covers and generalizes a recent result proved by Abu-Omar and Kittaneh in \cite{A.K.2}.

\begin{theorem}\label{001}
Let $T,S\in \mathcal{B}_{A}(\mathcal{H})$. Then,
\begin{align}\label{eq221}
 &\omega_A\left(T+S\right)\nonumber\\
& \leq \frac{1}{2}\left[ \omega_A\left(T\right)+\omega_A\left(S\right) +%
\sqrt{\left( \omega_A\left(T\right)-\omega_A\left(S\right) \right)
^{2}+4\sup_{\theta\in \mathbb{R}}\left\Vert \Re_A(e^{i\theta} T)\Re_A(e^{i\theta} S)\right\Vert_A}\right]\nonumber\\
 &\leq\omega_A\left(T\right)+\omega_A\left(S\right).
\end{align}%
\end{theorem}
\begin{proof}
Let $\theta \in \mathbb{R}$. It can be seen that $\Re_A[e^{i\theta}(T+S)]$ is an $A$-selfadjoint operator. So, by Lemma \ref{lem03} we get,
$$\left\|\Re_A[e^{i\theta}(T+S)]\right\|_A=r_A\Big(\Re_A[e^{i\theta}(T+S)]\Big).$$
By letting $T_1=I$, $S_1=\Re_A(e^{i\theta}T)$, $T_2=\Re_A(e^{i\theta}S)$ and $S_2=I$ in Lemma \ref{main2jd} and then using the norm monotonicity of matrices with nonnegative entries we get,
\begin{align*}
&\left\|\Re_A[e^{i\theta}(T+S)]\right\|_A\\
 &=r_A\Big(\Re_A(e^{i\theta}T)+\Re_A(e^{i\theta}S)\Big)\\
&\leq \left\| \begin{pmatrix}
\|\Re_A(e^{i\theta}T)\|_A &\left\|\Re_A(e^{i\theta}T)\Re_A(e^{i\theta}S)\right\|_A^{1/2}\\
\left\|\Re_A(e^{i\theta}T)\Re_A(e^{i\theta}S)\right\|_A^{1/2} &\|\Re_A(e^{i\theta}S)\|_A
\end{pmatrix}\right\|\\
&\leq \left\| \begin{pmatrix}
\omega_A(T) &\sqrt{\displaystyle\sup_{\theta\in \mathbb{R}}\left\|\Re_A(e^{i\theta}T)\Re_A(e^{i\theta}S)\right\|_A}\\
\sqrt{\displaystyle\sup_{\theta\in \mathbb{R}}\left\|\Re_A(e^{i\theta}T)\Re_A(e^{i\theta}S)\right\|_A} &\omega_A(S)
\end{pmatrix}\right\|\\
&=\frac{1}{2}\left[ \omega_A\left(T\right)+\omega_A\left(S\right) +%
\sqrt{\left( \omega_A\left(T\right)-\omega_A\left(S\right) \right)
^{2}+4\sup_{\theta\in \mathbb{R}}\left\Vert \Re_A(e^{i\theta} T)\Re_A(e^{i\theta} S)\right\Vert_A}\right].
\end{align*}
By taking supremum over all $\theta\in \mathbb{R}$ we get,
\begin{align}\label{refinewome}
 &\omega_A\left(T+S\right)\nonumber\\
& \leq \frac{1}{2}\left[ \omega_A\left(T\right)+\omega_A\left(S\right) +%
\sqrt{
\left( \omega_A\left(T\right)-\omega_A\left(S\right) \right)^{2}+4\sup_{\theta\in \mathbb{R}}\left\Vert \Re_A(e^{i\theta} T)\Re_A(e^{i\theta} S)\right\Vert_A}\right].
\end{align}%
This proves the first inequality. Moreover,
\begin{align*}
&\sqrt{\left( \omega_A\left(T\right)-\omega_A\left(S\right) \right)^{2}+4\sup_{\theta\in \mathbb{R}}\left\Vert \Re_A(e^{i\theta} T)\Re_A(e^{i\theta} S)\right\Vert_A}\\
&\leq \sqrt{\left( \omega_A\left(T\right)-\omega_A\left(S\right) \right)^{2}+4\omega_A\left(T\right)\omega_A\left(S\right)}\\
 &=\sqrt{\left( \omega_A\left(T\right)+\omega_A\left(S\right) \right)^{2}}=\omega_A\left(T\right)+\omega_A\left(S\right).
\end{align*}
So, by using \eqref{refinewome} we easily get the second inequality.
\end{proof}

The following lemma (see in \cite{saddi}) plays a crucial role in proving our next result.

\begin{lemma} \label{Lemma:3}
	Let $x, y, e \in \mathcal{H}$ with $\|e\|_A = 1.$ Then,
	\begin{eqnarray*}
		|\langle x, e \rangle_A \langle e, y \rangle_A| \leq \frac{1}{2}\big( |\langle x, y \rangle_A| + \|x\|_A\|y\|_A\big).
	\end{eqnarray*}
\end{lemma}

Now, we prove the following theorem.

\begin{theorem}\label{mai9}
Let $T,S\in \mathcal{B}_A(\mathcal{H})$. Then
\begin{align*}\label{derive1}
\omega_A(T+S)
&\leq \sqrt{\omega_A^2(T)+\omega_A^2(S)+\frac{1}{2}\left\|T^{\sharp_A}T+SS^{\sharp_A}\right\|_A+\omega_A(ST)}.
\end{align*}
\end{theorem}

\begin{proof}
Let $x\in \mathcal{H}$ be such that $\|x\|_A=1$. One can verify that
\begin{align*}
|\langle (T+S)x, x\rangle_A|^2
&\leq|\langle Tx, x\rangle_A|^2+|\langle Sx, x\rangle_A|^2+2|\langle Tx, x\rangle_A|\;|\langle Sx, x\rangle_A|\\
&=|\langle Tx, x\rangle_A|^2+|\langle Sx, x\rangle_A|^2+2|\langle Tx, x\rangle_A|\;|\langle x, S^{\sharp_A}x\rangle_A|.
\end{align*}
Using Lemma \ref{Lemma:3} we get,
\begin{align*}
&|\langle (T+S)x, x\rangle_A|^2\\
&\leq|\langle Tx, x\rangle_A|^2+|\langle Sx, x\rangle_A|^2+\|Tx\|_A\|S^{\sharp_A}x\|_A+|\langle Tx, S^{\sharp_A}x\rangle_A|\\
&=|\langle Tx, x\rangle_A|^2+|\langle Sx, x\rangle_A|^2+\sqrt{\langle T^{\sharp_A}Tx, x\rangle_A\langle SS^{\sharp_A}x, x\rangle_A}+|\langle STx, x\rangle_A|.
%&\leq|\langle Tx, x\rangle_A|^2+|\langle Sx, x\rangle_A|^2+\sqrt{\langle T^{\sharp_A}Tx, x\rangle_A\langle SS^{\sharp_A}x, x\rangle_A}+|\langle STx, x\rangle_A|.
\end{align*}
By using the arithmetic-geometric mean inequality we get,
\begin{align*}
|\langle (T+S)x, x\rangle_A|^2
&\leq\omega_A^2(T)+\omega_A^2(S)+\frac{1}{2}\left(\langle T^{\sharp_A}Tx, x\rangle_A+\langle SS^{\sharp_A}x, x\rangle_A\right)+\omega_A(ST)\\
&=\omega_A^2(T)+\omega_A^2(S)+\frac{1}{2}\langle (T^{\sharp_A}T+ SS^{\sharp_A})x, x\rangle_A+\omega_A(ST)\\
&\leq\omega_A^2(T)+\omega_A^2(S)+\frac{1}{2}\omega_A \left(T^{\sharp_A}T+ SS^{\sharp_A}\right)+\omega_A(ST)\\
&=\omega_A^2(T)+\omega_A^2(S)+\frac{1}{2}\left\|T^{\sharp_A}T+ SS^{\sharp_A}\right\|_A+\omega_A(ST),
\end{align*}
where the last equality follows from Lemma \ref{lem03}. So, we infer that
\begin{align*}
|\langle (T+S)x, x\rangle_A|^2\leq \omega_A^2(T)+\omega_A^2(S)+\frac{1}{2}\left\|T^{\sharp_A}T+ SS^{\sharp_A}\right\|_A+\omega_A(ST),
\end{align*}
for all $x\in \mathcal{H}$ with $\|x\|_A=1$. Thus, by taking the supremum over all $x\in \mathcal{H}$ with $\|x\|_A=1$, we get
\begin{align*}
\omega_A^2(T+S)\leq \omega_A^2(T)+\omega_A^2(S)+\frac{1}{2}\left\|T^{\sharp_A}T+ SS^{\sharp_A}\right\|_A+\omega_A(ST).
\end{align*}
This proves the desired result.
\end{proof}

As an application of the above theorem, we get the following corollary. %which provides an improvement of the second inequality in \eqref{kit}.

\begin{corollary}\label{corr2020}
Let $T\in\mathcal{B}_{A}(\mathcal{H})$. Then
\begin{align*}
\omega_A(T) \leq \frac{1}{2}\sqrt{{\big\|TT^{\sharp_A} + T^{\sharp_A} T\big\|}_A + 2\omega_A(T^2)}\le \frac{\sqrt{2}}{2}\sqrt{\|T^{\sharp_A} T+TT^{\sharp_A}\|_A}.
\end{align*}
\end{corollary}
\begin{proof}
Clearly, the first inequality follows by taking $S=T$ in Theorem \ref{mai9}. Moreover, it is well-known that $\omega_A(T^2)\leq \omega_A^2(T)$ (see \cite{feki01}) and $\omega_A^2(T) \leq \frac{1}{2}{\big\|TT^{\sharp_A} + T^{\sharp_A} T\big\|}_A$. So, we get that
\begin{align*}
\frac{1}{4}{\big\|TT^{\sharp_A} + T^{\sharp_A} T\big\|}_A + \frac{1}{2}\omega_A(T^2)
&\leq \frac{1}{4}{\big\|TT^{\sharp_A} + T^{\sharp_A} T\big\|}_A + \frac{1}{2}\omega_A^2(T)\\
&\leq \frac{1}{4}{\big\|TT^{\sharp_A} + T^{\sharp_A} T\big\|}_A + \frac{1}{4}{\big\|TT^{\sharp_A} + T^{\sharp_A} T\big\|}_A\\ %\quad \\ %(\text{by }\; \eqref{kit})\\
&=\frac{1}{2}{\big\|TT^{\sharp_A} + T^{\sharp_A} T\big\|}_A.
\end{align*}
This proves that the second inequality in Corollary \ref{corr2020}.
\end{proof}

\begin{remark}
Note that Corollary \ref{corr2020} has been recently proved in \cite{zamani1}.
\end{remark}

Our next improvement reads as:

\begin{theorem}\label{th-norm}
Let $T,S\in \mathcal{B}_A(\mathcal{H})$ be $A$-selfadjoint. Then,
\[\omega_A(T+S)\leq \sqrt{\omega_A^2(T+{\rm i}S)+\omega_A(ST)+\|T\|_A\|S\|_A} \leq \omega_A(T)+\omega_A(S).\]
\end{theorem}

\begin{proof}
Let $x\in \mathcal{H}$ be such that $\|x\|_A=1.$ Then we have,
\begin{eqnarray*}
|\langle (T+S)x,x \rangle_A|^2&\leq& (|\langle Tx,x\rangle_A|+|\langle Tx,x\rangle_A|)^2\\
&=& |\langle Tx,x\rangle_A|^2+|\langle Sx,x\rangle_A|^2+2|\langle Tx,x\rangle_A||\langle Sx,x\rangle_A|\\
&=& |\langle Tx,x\rangle_A +{\rm i}\langle Sx,x\rangle_A|^2+2|\langle Tx,x\rangle_A \langle Sx,x\rangle_A|\\
&=& |\langle (T+{\rm i}S)x,x\rangle_A|^2+2|\langle Tx,x\rangle_A \langle x,S^{\sharp_A}x\rangle_A|\\
&\leq & |\langle (T+{\rm i}S)x,x\rangle_A|^2+ \|Tx\|_A\|S^{\sharp_A}x\|_A+|\langle Tx,S^{\sharp_A}x\rangle_A|\\
&&\,\,\,\,\,\,\,\,\,\,\,\,\,\,\,\,\,\,\,\,\,\,\,\,\,\,\,\,\,\,\,\,\,\,\,\,\,\,\,\,\,\,\,\,\,\,\,\,\,\,\,\,\,\,\,\,\,\,\,\,\,\,\,\,\,\,\,\,\,\,\,\,\, ~~(\mbox{by Lemma \ref{Lemma:3}})\\
&= & |\langle (T+{\rm i}S)x,x\rangle_A|^2+ \|Tx\|_A\|S^{\sharp_A}x\|_A+|\langle STx,x\rangle_A|\\
&\leq& \omega_A^2(T+{\rm i}S)+\|T\|_A\|S\|_A+\omega_A(ST).
\end{eqnarray*}
Taking supremum over all $x\in \mathcal{H}$ with $\|x\|_A=1$ we get,
\begin{eqnarray*}
\omega_A^2(T+S) &\leq& \omega_A^2(T+{\rm i}S)+\|T\|_A\|S\|_A+\omega_A(ST).
\end{eqnarray*}
Thus, we have the first inequality of the theorem. Now we prove the second inequality. It is not dificult to show that $\omega_A^2(T+{\rm i}S)\leq \|T\|_A^2+\|S\|_A^2.$ Also we have $\omega_A(ST)\leq \|T\|_A\|S\|_A.$ So, $\omega_A^2(T+{\rm i}S)+\|T\|_A\|S\|_A+\omega_A(ST)\leq (\|T\|_A+\|S\|_A)^2.$ Since $\omega_A(T)=\|T\|_A$ and $\omega_A(S)=\|S\|_A$, so we get the required second inequality of the theorem.
\end{proof}

Next we obtain the inequalities for the sum of $k$ operators. First we recall the following results:
For $T\in \mathcal{B}(\mathcal{H})$, it was shown in \cite[Proposition 3.6.]{acg3} that $T\in \mathcal{B}_{A^{1/2}}(\mathcal{H})$ if and only if there exists a unique $\widetilde{T}\in \mathcal{B}(\mathbf{R}(A^{1/2}))$ such that $Z_AT =\widetilde{T}Z_A$. Here, $Z_{A}: \mathcal{H} \rightarrow \mathbf{R}(A^{1/2})$ is	 defined by $Z_{A}x = Ax$. Also, it has been proved in \cite{feki01} that for every $T\in \mathcal{B}_{A^{1/2}}(\mathcal{H})$ we have,
\begin{equation}\label{tilde}
\|T\|_A=\|\widetilde{T}\|_{\mathcal{B}(\mathbf{R}(A^{1/2}))}\quad\text{ and }\quad \omega_A(T)=\omega(\widetilde{T}).
\end{equation}
On the basis of the above results we obtain the following theorems.

\begin{theorem}\label{theo1}
For $i=1,2,\ldots,k,$ let $S_i\in \mathcal{B}_A(\mathcal{H})$. Then,
\begin{eqnarray*}
&& \omega^{4n}_A\left({ \sum_{i=1}^k S_i} \right) \\
&& \leq \frac{k^{4n-1}}{4} \left[  \left \| {  \sum_{i=1}^k \left ( \left ( S_i^{\sharp_A} S_i \right) ^{2n}+\left ( S_i S_i^{\sharp_A} \right)^{2n} \right) }   \right \|_A  + 2 \sum_{i=1}^k\omega_A\left({ \left ( S_i^{\sharp_A} S_i \right) ^{n}\left ( S_i S_i^{\sharp_A} \right)^{n}} \right) \right],
\end{eqnarray*}
for all $n=1,2,3,\ldots.$

%\begin{eqnarray*}
%\omega_A^{4n}\left(\sum_{i=1}^kS_i \right) \leq \frac{k^{4n-1}}{4} \left\|  \sum_{i=1}^k \left( |S_i|_A^{4n}+ |S_i^{\sharp_A}|_A^{4n} \right)   \right\|_A
% + \frac{k^{4n-1}}{2} \sum_{i=1}^k \omega_A \left( |S_i|_A^{2n} |S_i^{\sharp_A}|_A ^{2n}\right),
%\end{eqnarray*}
%where  $|S_i|_A^2=S_i^{\sharp_A}S_i$, $|S_i^{\sharp_A}|_A ^{2}=S_iS_i^{\sharp_A}$ and for all $n=1,2,3,\ldots.$
\end{theorem}

\begin{proof}
Let $x\in \mathcal{H}$ be such that $\|x\|=1$. Since $S_i\in \mathcal{B}_A(\mathcal{H})$, so $S_i\in \mathcal{B}(\mathcal{H})$. Then we have,
\begin{eqnarray*}
&& \left|  \left \langle \left(\sum_{i=1}^k S_i \right)x,x  \right \rangle    \right|^{4n}\\
&\leq&  \left(  \sum_{i=1}^k |\left \langle  S_i x,x  \right \rangle |   \right)^{4n}\\
&\leq&  k^{4n-1}  \sum_{i=1}^k |\left \langle  S_i x,x  \right \rangle |^{4n}\\
&\leq&  k^{4n-1}  \sum_{i=1}^k  \langle  |S_i| x,x \rangle^{2n} \langle  |S_i^*| x,x  \rangle^{2n},\,\,~~\Big(|\left \langle  S_i x,x  \right \rangle |^{2}\leq   \langle  |S_i| x,x \rangle \langle  |S_i^*| x,x  \rangle\Big)\\
&\leq&  k^{4n-1}  \sum_{i=1}^k  \langle  |S_i|^{2n} x,x  \rangle \langle  |S_i^*|^{2n} x,x  \rangle, \,\,~~\Big(\langle Sx,x \rangle^r\leq \langle S^r x,x \rangle, S\geq 0, r\geq 1 \Big) \\
&=&  k^{4n-1}  \sum_{i=1}^k  \langle  |S_i|^{2n} x,x  \rangle \langle   x, |S_i^*|^{2n}x  \rangle \\
&\leq&  \frac{k^{4n-1}}{2}  \sum_{i=1}^k  \left( \left \||S_i|^{2n} x \right \| \left \||S_i^*|^{2n}x \right\|+   |\langle  |S_i|^{2n} x,  |S_i^*|^{2n}x  \rangle| \right)\\
&\leq&  \frac{k^{4n-1}}{2}  \sum_{i=1}^k  \left( \frac{1}{2} (\left \||S_i|^{2n} x \right \|^2+ \left \||S_i^*|^{2n}x \right\|^2)+   |\langle  |S_i|^{2n}|S_i^*|^{2n} x,  x  \rangle| \right)\\
&=&  \frac{k^{4n-1}}{4}  \sum_{i=1}^k \left \langle (|S_i|^{4n}+|S_i^*|^{4n}) x,x \right \rangle +\frac{k^{4n-1}}{2} \sum_{i=1}^k |\langle  |S_i|^{2n}|S_i^*|^{2n} x,  x  \rangle| \\
&=&  \frac{k^{4n-1}}{4}   \left \langle \left(\sum_{i=1}^k(|S_i|^{4n}+|S_i^*|^{4n})\right) x,x \right \rangle +\frac{k^{4n-1}}{2} \sum_{i=1}^k |\langle  |S_i|^{2n}|S_i^*|^{2n} x,  x  \rangle| \\
&\leq &  \frac{k^{4n-1}}{4}     \left \| \sum_{i=1}^k(|S_i|^{4n}+|S_i^*|^{4n})  \right \| + \frac{k^{4n-1}}{2}\sum_{i=1}^k\omega\left(  |S_i|^{2n}|S_i^*|^{2n} \right),\\
\end{eqnarray*}
where the second inequality follows from Bohr's inequality $($\cite{v}$)$, i.e.,
if for $i=1,2,\ldots,n$, $a_i$ be a positive real number then
\[\left( \sum_{i=1}^ka_i\right)^r \leq k^{r-1}\sum_{i=1}^ka_i^r, r\geq 1\]
and the fifth inequality follows from  Buzano's inequality (\cite{buzano}), i.e.,
if $x,y,e\in \mathcal{H}$ with $\|e\|=1$ then
\[|\langle x,e\rangle \langle e,y\rangle|\leq \frac{1}{2}\left(\|x\| \|y\|+|\langle x,y\rangle|\right).\]
Taking supremum over all $x\in \mathcal{H}$ with $\|x\|=1$ we get,
\begin{eqnarray}\label{H1}
 \omega^{4n}\left(\sum_{i=1}^kS_i \right) \leq   \frac{k^{4n-1}}{4}     \left \| \sum_{i=1}^k(|S_i|^{4n}+|S_i^*|^{4n})  \right \| + \frac{k^{4n-1}}{2}\sum_{i=1}^k\omega\left(  |S_i|^{2n}|S_i^*|^{2n} \right).
\end{eqnarray}
Now since $\mathcal{B}_A(\mathcal{H}) \subseteq \mathcal{B}_{A{^{1/2}}}(\mathcal{H})$, so for each $i=1,2,\ldots,k$, $S_i\in \mathcal{B}_{A{^{1/2}}}(\mathcal{H}).$Therefore, there exists unique $\widetilde{S_i}$ in $\mathcal{B}(\mathbf{R}(A^{1/2}))$ such that $Z_AS_i =\widetilde{S_i}Z_A$. Now, $\mathbf{R}(A^{1/2})$ being a complex Hilbert space, so (\ref{H1}) implies that
\begin{eqnarray*}
 && \omega^{4n}\left( \sum_{i=1}^k\widetilde{S_i} \right) \\
&& \leq   \frac{k^{4n-1}}{4}  \left \| \sum_{i=1}^k(|\widetilde{S_i}|^{4n}+|\widetilde{S_i}^*|^{4n})  \right \|_{\mathcal{B}(\mathbf{R}(A^{1/2}))} +\frac{k^{4n-1}}{2}  \sum_{i=1}^k \omega\left(  |\widetilde{S_i}|^{2n}|\widetilde{S_i}^*|^{2n} \right).
\end{eqnarray*}
It is well-known that for $S,T \in \mathcal{B}_{A{^{1/2}}}(\mathcal{H})$, we have $\widetilde{S+\lambda T}=\widetilde{S}+\lambda \widetilde{T}$ and $\widetilde{ST}=\widetilde{S}\widetilde{T}$ for all $\lambda \in \mathbb{C}$ (see \cite{fekilaa}). So, from the above inequality we have,
\begin{eqnarray*}
 \omega^{4n}\left(   \widetilde{ \sum_{i=1}^k S_i} \right)
&\leq&    \frac{k^{4n-1}}{4}   \left \| \sum_{i=1}^k \left( \left ((\widetilde{S_i})^*\widetilde{S_i} \right) ^{2n}+\left ( \widetilde{S_i} (\widetilde{S_i})^*\right)^{2n} \right) \right \|_{\mathcal{B}(\mathbf{R}(A^{1/2}))}\\
&& + \frac{k^{4n-1}}{2} \sum_{i=1}^k \omega\left(  \left ((\widetilde{S_i})^*\widetilde{S_i} \right)^{n}\left ( \widetilde{S_i} (\widetilde{S_i})^*\right)^{n} \right).
\end{eqnarray*}
Also since $(\widetilde{S_i})^*=\widetilde{S_i^{\sharp_A}}$, so
\begin{eqnarray*}
 \omega^{4n}\left(   \widetilde{ \sum_{i=1}^k S_i} \right)
&\leq&    \frac{k^{4n-1}}{4}  \left \|\sum_{i=1}^k \left( \left (\widetilde{S_i^{\sharp_A}}\widetilde{S_i} \right) ^{2n}+\left ( \widetilde{S_i} \widetilde{S_i^{\sharp_A}}\right)^{2n} \right) \right \|_{\mathcal{B}(\mathbf{R}(A^{1/2}))} \\
&& + \frac{k^{4n-1}}{2}\omega\left(  \left (\widetilde{S_i^{\sharp_A}}\widetilde{S_i} \right)^{n}\left ( \widetilde{S_i} \widetilde{S_i^{\sharp_A}}\right)^{n} \right) \\
&=& \frac{k^{4n-1}}{4}  \left \| \widetilde{  \sum_{i=1}^k \left ( \left ( S_i^{\sharp_A} S_i \right) ^{2n}+\left ( S_i S_i^{\sharp_A} \right)^{2n} \right) }   \right \|_{\mathcal{B}(\mathbf{R}(A^{1/2}))} \\
&& + \frac{k^{4n-1}}{2} \sum_{i=1}^k\omega\left(  \widetilde{ \left ( S_i^{\sharp_A} S_i \right) ^{n}\left ( S_i S_i^{\sharp_A} \right)^{n}} \right).
\end{eqnarray*}
Hence,
\begin{eqnarray*}
\omega^{4n}_A\left({ \sum_{i=1}^k S_i} \right) &\leq& \frac{k^{4n-1}}{4}  \left \| {  \sum_{i=1}^k \left ( \left ( S_i^{\sharp_A} S_i \right) ^{2n}+\left ( S_i S_i^{\sharp_A} \right)^{2n} \right) }   \right \|_A \\
&& + \frac{k^{4n-1}}{2} \sum_{i=1}^k\omega_A\left({ \left ( S_i^{\sharp_A} S_i \right) ^{n}\left ( S_i S_i^{\sharp_A} \right)^{n}} \right).
\end{eqnarray*}
Thus, we complete the proof.
\end{proof}

In particular considering $k=1$ and $n=1$ in Theorem \ref{theo1} we get the following result.
\begin{corollary}
Let $S\in \mathcal{B}_A(\mathbb{H}).$ Then
\[\omega_A^{4}(S) \leq \frac{1}{4}\left \|  \left ( S^{\sharp_A} S \right) ^{2}+\left ( S S^{\sharp_A} \right)^{2}     \right \|_A  + \frac{1}{2} \omega_A\left( S^{\sharp_A} S^2 S^{\sharp_A} \right). \]

\end{corollary}

Next result reads as:

\begin{theorem}\label{theo2}
For $i=1,2,\ldots,n,$ let $S_i\in \mathcal{B}_A(\mathcal{H})$. Then
\[\omega_A^{2n} \left(\sum_{i=1}^kS_i  \right)\leq \frac{k^{2n-1}}{2}\left\|  \sum_{i=1}^k \left( \left (S_i^{\sharp_A}S_i\right)^n+  \left( S_iS_i^{\sharp_A}\right)^n\right)  \right\|_A.\]
\end{theorem}

\begin{proof}
Let $x\in \mathcal{H}$ with $\|x\|=1$. $S_i\in \mathcal{B}_A(\mathcal{H})$ implies $S_i\in \mathcal{B}(\mathcal{H})$. So we have,
\begin{eqnarray*}
&& \left|  \left \langle \left(\sum_{i=1}^k S_i \right)x,x  \right \rangle    \right|^{2n}\\
&\leq&  \left(  \sum_{i=1}^k |\left \langle  S_i x,x  \right \rangle |   \right)^{2n}\\
&\leq&  k^{2n-1}  \sum_{i=1}^k |\left \langle  S_i x,x  \right \rangle |^{2n}\\
&\leq&  k^{2n-1}  \sum_{i=1}^k  \langle  |S_i| x,x \rangle^{n} \langle  |S_i^*| x,x  \rangle^{n},\,\,~~\Big(|\left \langle  S_i x,x  \right \rangle |^{2}\leq   \langle  |S_i| x,x \rangle \langle  |S_i^*| x,x  \rangle\Big)\\
&\leq&  k^{2n-1}  \sum_{i=1}^k  \langle  |S_i|^{n} x,x  \rangle \langle  |S_i^*|^{n} x,x  \rangle, \,\,~~\Big(\langle Sx,x \rangle^r\leq \langle S^r x,x \rangle, S\geq 0, r\geq 1 \Big) \\
&\leq& \frac{k^{2n-1}}{2} \sum_{i=1}^k  \left(\langle  |S_i|^{n} x,x  \rangle^2+ \langle  |S_i^*|^{n} x,x  \rangle^2\right)\\
&\leq& \frac{k^{2n-1}}{2} \sum_{i=1}^k  \left(\langle  |S_i|^{2n} x,x  \rangle+ \langle  |S_i^*|^{2n} x,x  \rangle\right)\\
&=& \frac{k^{2n-1}}{2} \sum_{i=1}^k  \langle  \left(|S_i|^{2n}+ |S_i^*|^{2n} \right) x,x  \rangle\\
&=& \frac{k^{2n-1}}{2}   \left \langle \sum_{i=1}^k\left( |S_i|^{2n}+ |S_i^*|^{2n} \right) x,x  \right \rangle\\
&\leq& \frac{k^{2n-1}}{2}   \left \| \sum_{i=1}^k\left( |S_i|^{2n}+ |S_i^*|^{2n} \right) \right \|.
\end{eqnarray*}
Taking supremum over all $x\in \mathcal{H}$ with $\|x\|=1$ we get,
\begin{eqnarray}\label{H2}
\omega^{2n} \left(\sum_{i=1}^kS_i  \right) &\leq& \frac{k^{2n-1}}{2}   \left \| \sum_{i=1}^k\left( |S_i|^{2n}+ |S_i^*|^{2n} \right)\right \|.
\end{eqnarray}
Now for each $i=1,2,\ldots,k$, $S_i\in \mathcal{B}_{A{^{1/2}}}(\mathcal{H}).$ So, there exists unique $\widetilde{S_i}$ in $\mathcal{B}(\mathbf{R}(A^{1/2}))$ such that $Z_AS_i =\widetilde{S_i}Z_A$. Now $\mathbf{R}(A^{1/2})$ being a complex Hilbert space, we have from (\ref{H2}) that
\begin{eqnarray*}
 \omega^{2n}\left( \sum_{i=1}^k\widetilde{S_i} \right) &\leq&   \frac{k^{2n-1}}{2}  \left \| \sum_{i=1}^k(|\widetilde{S_i}|^{2n}+|\widetilde{S_i}^*|^{2n})  \right \|_{\mathcal{B}(\mathbf{R}(A^{1/2}))}.
\end{eqnarray*}
Using the property that if $S,T \in \mathcal{B}_{A{^{1/2}}}(\mathcal{H})$ then $\widetilde{S+\lambda T}=\widetilde{S}+\lambda \widetilde{T}$ and $\widetilde{ST}=\widetilde{S}\widetilde{T}$ for all $\lambda \in \mathbb{C}$, we have
\begin{eqnarray*}
 \omega^{2n}\left(   \widetilde{ \sum_{i=1}^k S_i} \right)
&\leq&    \frac{k^{2n-1}}{2}   \left \| \sum_{i=1}^k \left( \left ((\widetilde{S_i})^*\widetilde{S_i} \right) ^{n}+\left ( \widetilde{S_i} (\widetilde{S_i})^*\right)^{n} \right) \right \|_{\mathcal{B}(\mathbf{R}(A^{1/2}))}.
\end{eqnarray*}
Also $(\widetilde{S_i})^*=\widetilde{S_i^{\sharp_A}}$, so
\begin{eqnarray*}
 \omega^{2n}\left(   \widetilde{ \sum_{i=1}^k S_i} \right)
&\leq&    \frac{k^{2n-1}}{2}  \left \|\sum_{i=1}^k \left( \left (\widetilde{S_i^{\sharp_A}}\widetilde{S_i} \right) ^{n}+\left ( \widetilde{S_i} \widetilde{S_i^{\sharp_A}}\right)^{n} \right) \right \|_{\mathcal{B}(\mathbf{R}(A^{1/2}))}\\
&=& \frac{k^{2n-1}}{2}  \left \| \widetilde{  \sum_{i=1}^k \left ( \left ( S_i^{\sharp_A} S_i \right) ^{n}+\left ( S_i S_i^{\sharp_A} \right)^{n} \right) }   \right \|_{\mathcal{B}(\mathbf{R}(A^{1/2}))}.
\end{eqnarray*}
Hence,
\begin{eqnarray*}
\omega^{2n}_A\left({ \sum_{i=1}^k S_i} \right) &\leq& \frac{k^{2n-1}}{2}  \left \| {  \sum_{i=1}^k \left ( \left ( S_i^{\sharp_A} S_i \right) ^{n}+\left ( S_i S_i^{\sharp_A} \right)^{n} \right) }   \right \|_A.
\end{eqnarray*}
Hence we complete the proof.
\end{proof}

Finally we obtain the following result.

\begin{theorem}\label{theo3}
For $i=1,2,\ldots,n,$ let $S_i\in \mathcal{B}_A(\mathcal{H})$. Then,
\[\omega_A^{2n} \left(\sum_{i=1}^kS_i  \right)\leq \frac{k^{2n-1}}{\sqrt{2}}  \sum_{i=1}^k \omega_A\left( \left (S_i^{\sharp_A}S_i\right)^n+  {\rm i}\left( S_iS_i^{\sharp_A}\right)^n\right).\]
\end{theorem}

\begin{proof}
Let $x\in \mathcal{H}$ with $\|x\|=1$. Then we have,
\begin{eqnarray*}
\left|  \left \langle \left(\sum_{i=1}^k S_i \right)x,x  \right \rangle    \right|^{2n}
&\leq& \frac{k^{2n-1}}{2} \sum_{i=1}^k  \left(\langle  |S_i|^{2n} x,x  \rangle+ \langle  |S_i^*|^{2n} x,x  \rangle\right).
\end{eqnarray*}
Now we observe that $|a+b|\leq \sqrt{2}|a+{\rm i}b|$ for all $a,b\in \mathbb{R}.$ Using this inequality we have,
\begin{eqnarray*}
\left|  \left \langle \left(\sum_{i=1}^k S_i \right)x,x  \right \rangle    \right|^{2n}
&\leq& \frac{k^{2n-1}}{\sqrt{2}} \sum_{i=1}^k  \left|\langle  |S_i|^{2n} x,x  \rangle+ {\rm i}\langle  |S_i^*|^{2n} x,x  \rangle\right|\\
&\leq& \frac{k^{2n-1}}{\sqrt{2}} \sum_{i=1}^k  \omega\left (  |S_i|^{2n} + {\rm i}  |S_i^*|^{2n} \right).
\end{eqnarray*}
Taking supremum over all $x\in \mathcal{H}$ with $\|x\|=1$ we get,
\begin{eqnarray}\label{H3}
\omega^{2n} \left(\sum_{i=1}^kS_i  \right) &\leq& \frac{k^{2n-1}}{\sqrt{2}} \sum_{i=1}^k  \omega\left (  |S_i|^{2n} + {\rm i}  |S_i^*|^{2n} \right).
\end{eqnarray}
Now for each $i=1,2,\ldots,k$, $S_i\in \mathcal{B}_{A{^{1/2}}}(\mathcal{H}).$ So, there exists unique $\widetilde{S_i}$ in $\mathcal{B}(\mathbf{R}(A^{1/2}))$ such that $Z_AS_i =\widetilde{S_i}Z_A$. Now $\mathbf{R}(A^{1/2})$ being a complex Hilbert space, from (\ref{H3}) we get,
\begin{eqnarray}\label{H4}
\omega^{2n} \left(\sum_{i=1}^k\widetilde{S_i}  \right) &\leq& \frac{k^{2n-1}}{\sqrt{2}} \sum_{i=1}^k  \omega\left (  |\widetilde{S_i}|^{2n} + {\rm i} |\widetilde{S_i}^*|^{2n} \right).
\end{eqnarray}
Using the property that if $S,T \in \mathcal{B}_{A{^{1/2}}}(\mathcal{H})$ then $\widetilde{S+\lambda T}=\widetilde{S}+\lambda \widetilde{T}$ and $\widetilde{ST}=\widetilde{S}\widetilde{T}$ for all $\lambda \in \mathbb{C}$ and  $(\widetilde{S_i})^*=\widetilde{S_i^{\sharp_A}}$, we have form (\ref{H4}) that
\begin{eqnarray*}
\omega^{2n} \left(\widetilde{\sum_{i=1}^k\widetilde{S_i} } \right) &\leq& \frac{k^{2n-1}}{\sqrt{2}} \sum_{i=1}^k  \omega\left (  \widetilde{(S_i^{\sharp_A}S_i)^{n} + {\rm i}  (S_iS_i^{\sharp_A})^{n}} \right).
\end{eqnarray*}
\mbox{Hence,} ~~
 \[\omega_A^{2n} \left(\sum_{i=1}^kS_i  \right)\leq \frac{k^{2n-1}}{\sqrt{2}}  \sum_{i=1}^k \omega_A\left( \left (S_i^{\sharp_A}S_i\right)^n+  {\rm i}\left( S_iS_i^{\sharp_A}\right)^n\right),\] as required.
\end{proof}

The following corollary is an easy consequence of Theorem \ref{theo3}.

\begin{corollary}
Let $S\in \mathcal{B}_A(\mathcal{H})$. Then
\[\omega_A^{2} \left(S  \right)\leq \frac{1}{\sqrt{2}} \omega_A\left(  S^{\sharp_A}S+  {\rm i}~~ SS^{\sharp_A}\right).\]
\end{corollary}

%-----------------------------------------------{chapter}{Bibliography}------------------------------------------

\end{document}